\documentclass{ipi}
\usepackage{amsfonts, mathrsfs}
\usepackage{amssymb,amsmath}
\usepackage{amssymb}
  \usepackage{paralist}

  \usepackage{hyperref} 

\def\currentvolume{X}
 \def\currentissue{X}
  \def\currentyear{200X}
   
    \def\ppages{X--XX}

\newcommand{\ud}{\mathrm{d}}
\newcommand{\e}{\mathrm{e}}

\newcommand{\R}{\mathbb{R}}
\newcommand{\N}{\mathbb{N}}
\newcommand{\Z}{\mathbb{Z}}
\newcommand{\TT}{\mathbb{T}}
\newcommand{\QQ}{\mathbb{Q}}
\newcommand{\Qu}{\mathbb{Q}^u}

\newcommand{\cdiv}{\nabla\cdot}
\newcommand{\lt}{\left}
\newcommand{\rt}{\right}

\newcommand{\muy}{\mu^y}

\newcommand{\kk}{\tiny{\rm K}}

\newcommand{\bbR}{\mathbb{R}}
\newcommand{\cG}{\mathcal{G}}
\newcommand{\bbP}{\mathbb{P}}

\newcommand{\bbE}{\mathbb{E}}

\newcommand{\dhh}{d_{\mbox {\tiny{\rm Hell}}}}

\newcommand{\la}{\langle}
\newcommand{\ra}{\rangle}

\newtheorem{theorem}{Theorem}[section]

\newtheorem{proposition}{Proposition}

\theoremstyle{definition}

\newtheorem{remark}{Remark}
\newtheorem{asp}[theorem]{Assumption}

\title[Besov Priors for Bayesian Inverse problems]
      {Besov Priors for Bayesian Inverse problems}

\author[M. Dashti, S. Harris and A. M. Stuart]{}


\begin{document}
\maketitle

\centerline{\scshape Masoumeh Dashti }
\medskip
{\footnotesize
 \centerline{Mathematics Institute, University of Warwick}
   \centerline{Coventry CV4 7AL, UK}
} 

\medskip

\centerline{\scshape Stephen Harris}
\medskip
{\footnotesize
 \centerline{School of Mathematics, University of Edinburgh}
   \centerline{Edinburgh EH9 3JZ, UK}
} %

\medskip

\centerline{\scshape Andrew Stuart}
\medskip
{\footnotesize
 \centerline{Mathematics Institute, University of Warwick}
   \centerline{Coventry CV4 7AL, UK}
} %

\bigskip

 \centerline{(Communicated by the associate editor name)}

\begin{abstract}
We consider the inverse problem of estimating a function $u$
from noisy, possibly nonlinear, observations. 
We adopt a Bayesian approach to the problem. This
approach has a long history for inversion, dating back to 1970,
and has, over the last decade, gained importance as a practical tool.
However most of the existing theory has been
developed for Gaussian prior measures. Recently Lassas,
Saksman and Siltanen (Inv. Prob. Imag. 2009) 
showed how to construct Besov prior measures,
based on wavelet expansions with random coefficients, and used these
prior measures to study linear inverse problems. In this
paper we build on this development of Besov priors to include
the case of nonlinear measurements. In doing so a key technical tool, 
established here, is a Fernique-like theorem for 
Besov measures. This theorem enables us to 
identify appropriate conditions on the forward solution
operator which, when matched to properties of the prior
Besov measure, imply the well-definedness and well-posedness 
of the posterior measure.  We then consider the application  of these 
results to the inverse problem of finding the diffusion coefficient 
of an elliptic partial differential equation, given noisy 
measurements of its solution.  
\end{abstract}

\section{Introduction}\label{s:intro}

The Bayesian approach to inverse problems is an attractive one.
It mathematizes the way many practitioners
incorporate new data into their understanding of a given
phenomenon; and it results in a precise quantification of uncertainty.
Although this approach to inverse problems has a long history,
starting with the paper \cite{Fr70}, it is only in the last decade
that its use has become widespread as a computationl tool \cite{ks04}. 
The theoretical side of the subject, which is the focus of this
paper, is far from fully developed, with many interesting open questions. 
The work \cite{Fr70} concerned
linear Gaussian problems, and the mathematical underpinnings of
such problems were laid in the papers \cite{man83,leht89}.
An important theme in subsequent theoretical work concerning 
linear Gaussian problems has been to study the
effect of discretization, in both the state space and the data space,
and to identify approaches which give rise to meaningful limits
\cite{La02, La07}. In many imaging problems, the detection of
edges and interfaces is important and such problems are not well
modelled by Gaussian priors. This has led to two recent works
which try to circumvent this issue: the paper \cite{Las09}
introduces the notion of Besov priors, based on wavelet expansions,
and the paper \cite{HL11} uses heirarchical Gaussian models to
create a discretization of the Mumford-Shah penalization, in
one dimension. The thesis \cite{Pi05} studies a number of related
issues for quite general measurement models. A different series
of papers has studied the development of Bayesian inversion
with Gaussian priors and nonlinear measurement functions \cite{CDRS09,St10},
in which general criteria for a Bayes theorem, interpreted to
mean that the posterior distribution has density with respect to the
prior, are given. This framework has been used to study the effect
of approximation of both the space in which the prior lies and the
forward model \cite{CDS10,DaSt11}, allowing the transfer of error
estimates for the approximation of the forward problem, into
estimates for the approximation of the posterior measure. The goal
of the present paper is to extend this type of approximation theory
from Gaussian priors to the Besov priors introduced in \cite{Las09},
and in the case of nonlinear measurement functions.

We consider the noisy nonlinear operator equation
\begin{align}
y=\cG(u)+\eta
\label{eq:one}
\end{align}
with $\cG:X\to Y$, $X,Y$ Banach spaces and $\eta$ a $Y$-valued random variable.
We suppose in this paper that $y$, the operator $\cG$ and the statistical properties of $\eta$ are known, and an estimation of $u\in X$ is to be found.
Such an inverse problem appears in many practical situations where the function of interest (here $u$) cannot be observed directly and has to be obtained from other observable quantities and through the mathematical model relating them.

This problem is in general ill-posed and therefore to obtain a reasonable 
approximation of $u$ in a stable way, we need prior information 
about the solution \cite{EHN96,ks04}. 
In particular if we expect the unknown function to be sparse in some
specific orthonormal basis of $X$, implementing the prior information in a way that respects this sparsity will result in more efficient finite-dimensional approximation 
of the solution. For instance a smooth function with a few local irregularities has a more sparse expansion in a wavelet basis compared, for example, to a Fourier basis, and adopting 
regularization methods which respect the sparse 
wavelet expansion of 
the unknown function is of interest in many applications. 
Approximation techniques based on wavelet bases for recovering 
finite dimensional estimates of the unknown function are extensively studied 
in the approximation theory, signal processsing and statistics
literature; see for example \cite{AbS98,Cham98,DDD04,DoJ94}. 
The paper \cite{AbSS98} introduced a nonparametric Bayesian
approach to the problem of signal recovery, adopting
a clever posterior construction tuned to the simple form of the
observation operator. The paper \cite{Las09} considers more
general linear observation operators and constructs the wavelet-based
Besov prior; the construction is through a generalization of
the Karhunen-Lo\`ev\'e expansion to non-Gaussian coefficients and
wavelet bases.

We adopt a Bayesian approach to the above inverse problem, 
so that regularization is implicit in the prior measure on $X$. 
We study the Besov prior measure introduced in \cite{Las09}
as this measure is constructed so that it factors as the product of independent measures along members of a wavelet basis, see Section \ref{s:bBayes}. 
We first make sense of Bayes rule in an infinite dimensional
setting, and then study the finite dimensional approximations 
after having proved the well-posedness of the posterior over an infinite-dimensional Banach space.
Bayes rule for functions is here interpreted as follows.
We put a prior probability measure $\mu_0(\ud u)$ on $u$
and then specify the distribution of $y|u$ via \eqref{eq:one},
thereby defining a joint probability distribution on $(u,y).$
If the posterior measure $\mu^y(\ud u)=\bbP(\ud u|y)$
is absolutely continuous with respect to the prior measure 
$\mu_0(\ud u)=\bbP(\ud u)$ then
Bayes theorem is interpreted as the following 
formula for the Radon-Nikodym derivative:
\begin{equation}
\label{eq:radon1}
\frac{\ud\mu^y}{\ud\mu_0}(u)\propto\exp \bigl(-\Phi(u;y)\bigr).
\end{equation} 
Here $\Phi$ depends on the the specific instance of
the data $y$, the forward (or observation) operator
$\cG$ and the distribution of $\eta$  
(see Section \ref{s:bBayes}). 
For simplicity we work in the case where $X$
comprises periodic functions on the $d-$dimensional
torus $\TT^d$; generalizations are possible.

The problem of making sense of Bayes rule for nonlinear observation
operators, and with a Gaussian prior, is addressed in \cite{CDRS09,St10};
sufficient conditions on $\Phi$ and the Gaussian prior $\mu_0$ which imply 
the well-definedness and well-posedness of the posterior are obtained. 
Key to obtaining the conditions on $\Phi$ is the 
Fernique theorem for Gaussian measures which gives an upper bound 
on the growth of $\mu_0$-integrable functions. 
Our aim here is to generalize the results of  
\cite{CDRS09,St10} to the case of Besov prior measures and so
we need a similar Fernique-like result for the Besov measures.
In Section \ref{s:Besov-mu}, following \cite{Las09}, 
we construct the Besov measures using wavelet expansions with
i.i.d.\ random coefficients (Karhunen-Lo\`ev\'e expansions) 
of their draws, and prove a Fernique-like result for these 
measures (Theorem \ref{t:Fer}).
We then use this result in Section \ref{s:bBayes} to find the conditions on the operator $\cG$ which ensure the 
well-definedness (Theorem \ref{t:welldb}) and 
well-posedness (Theorem \ref{t:wellpb}) 
of the posterior measure $\mu^y$, 
provided that the Besov measure $\mu_0$ is chosen appropriately. 
We then build on this theory to quantify finite
dimensional approximation of the posterior measure,
culminating in Theorem \ref{t:wellp2b}.

In Section \ref{s:elliptic} we apply these results to
the inverse problem of finding the diffusion coefficient of a linear 
uniformly elliptic partial differential equation
in divergence form, in a bounded domain in dimension 
$d \le 3$, from measurements of the solution
in the interior.  Such an inverse problem emerges in geophysical applications where $u$ is the log-permeability of 
the subsurface. 
It is studied using a Bayesian approach on function space
in \cite{DaSt11} for Gaussian priors and in \cite{ScSt11}
for non-Gaussian priors which give rise to an almost sure
uniform lower bound on the permeability. In many
subsurface applications it is natural to expect that
$u$ is a smooth function with a few local irregularities, and
it is not natural to expect an almost sure uniform lower bound on
the permeability, nor is a Gaussian prior appropriate.
For these reasons a wavelet basis, and hence 
a Besov prior, provides a plausible candidate from a modelling
perspective. 
We show that the conditions we require for 
$\Phi$ in this problem hold naturally
when $X$ is the space of H\"older continuous functions $C^t$.
This also suggests 
the use of wavelet bases which are unconditional bases for $C^t$.
Thus for the elliptic inverse problem with Besov
priors we state a well-posedness result
for the posterior measure (Theorem \ref{thm:dwellposterior})
and we also quantify the effect of the finite dimensional approximation 
of the log-permiablility in the wavelet basis used in construction 
of the prior, on the posterior measure,
with result summarized in Theorem \ref{t:appw}.

\section{Besov measures}
\label{s:Besov-mu}
We develop Besov measures following the construction
in \cite{Las09}.
Let $\{\psi_l\}_{l=1}^\infty$ be a basis for $L^2(\TT^d)$, $\TT^d=(0,1]^d$, $d\le3$, so that any $f\in L^2(\TT^d)$ can be written as
$$
f(x)=\sum_{l=1}^\infty f_l\,\psi_l(x).
$$
Let $X^{s,q}$ 
be a Banach space with norm $\|\cdot\|_{X^{s,q}}$ defined as
$$
\|f\|_{X^{s,q}}=\left(  \sum_{l=1}^\infty l^{(sq/d+q/2-1)}|f_l |^q \right)^{1/q}
$$
with $q\ge 1$ and $s>0$. We now construct a probability
measure on functions by randomizing the coefficients of an
expansion in the basis $\{\psi_l\}_{l=1}^\infty$. 
The space $X^{s,q}$ will play a role analogous to
the Cameron-Martin space for this measure. Indeed
when $\{\psi_l\}_{l=1}^\infty$ is chosen to
be the Karhunen-Lo\`eve basis for a Gaussian measure and $q=2$,
our choice of coefficients will ensure that $X^{s,2}$ is
precisely the Cameron-Martin space.

Let $1\leq q <\infty$, $s>0$, and $\kappa>0$ be fixed and $\{\xi_l\}_{l=1}^\infty$ 
be real-valued i.i.d.\ 
random variables with probability density function
\begin{equation*}
\pi_\xi(x) \propto \exp(-\frac12|x|^q).
\end{equation*}
Let the random function $u$ be defined as follows
\begin{equation}\label{e:ranfun}
u(x) = \sum_{l=1}^\infty l^{-(\frac{s}{d}+ \frac{1}{2} -\frac{1}{q})}(\tfrac{1}{\kappa})^\frac{1}{q}\xi_l\psi_l(x).
\end{equation}
We will refer to the induced measure on functions $u$
as $\mu_0.$
We note that, since $\{\psi_l\}_{l=1}^\infty$ is an orthonormal basis and
\begin{equation*}
u(x) = \sum_{l=1}^\infty u_l\psi_l(x)
\end{equation*}
with $u_l=l^{-(\frac{s}{d}+ \frac{1}{2} -\frac{1}{q})}(\frac{1}{\kappa})^\frac{1}{q}\xi_l$, we have
\begin{align}
\prod_{l=1}^\infty \exp(-\frac12|\xi_l|^q) &=\prod_{l=1}^\infty \exp\left(-\frac{\kappa}{2}l^{\frac{qs}{d}+\frac{q}{2}-1}|u_l|^q\right)\nonumber\\
& = \exp\left(-\frac{\kappa}{2} \sum_{l=1}^{\infty} l^{\frac{qs}{d}+\frac{q}{2}-1}|u_l|^q\right)\nonumber\\
& = \exp(-\frac{\kappa}{2} \|u\|_{X^{s,q}}^q)\label{e:Xsp}
\end{align}
Thus, informally, $u$ has a Lebesgue density proportional to $ \exp(-\frac{\kappa}{2} \|u\|_{X^{s,q}}^q)$.
\noindent We say that $u$ is distributed according to an 
$X^{s,q}$ measure with parameter $\kappa$, or, briefly,
a $(\kappa, X^{s,q})$ measure.

 \begin{remark}\label{r:psi}
 If $\{\psi_l\}_{l=1}^\infty$ in (\ref{e:ranfun}) is an 
 $r$-regular wavelet basis
 \footnote{An $r$-regular wavelet basis for $L^2(\R^d)$ is a wavelet basis
 with $r$-regular scaling function and mother wavelets. A function $f$ is $r$-regular if $f\in C^r$ and $|\partial^\alpha f(x)|\le C_m(1+|x|)^{-m}$, for any integer
 $m\in \N$ and any multi-index $\alpha$ with
 $|\alpha|=\alpha_1+\dots+\alpha_d\le r$. 
 An $r$-regular basis for $L^2(\TT^d)$ is obtained by periodification of
  the $r$-regular wavelets of $L^2(\R^d)$ \cite{Dau92,Mey92}.}
 for $L^2(\TT^d)$, with $r>s$,
then $\|\cdot\|_{X^{s,q}}$ is the Besov $B^{s}_{qq}$ \cite{Tri83} norm and $u$
 is distributed according to  a $(\kappa, B^{s}_{qq})$ measure.
Furthermore, if $q=2$ with $\{\psi_l\}_{l=1}^\infty$ either a 
wavelet or Fourier basis,  we obtain a Gaussian measure 
with Cameron-Martin space $B^s_{22}$, which is
simply the Hilbert space $H^s=H^s(\TT^d)$. 
Indeed (\ref{e:ranfun}) reduces to 
\begin{equation*}
u(x) = \sqrt{\frac{1}{\kappa}}\sum_{l=1}^\infty l^{-\frac{s}{d}}\, \xi_l\, \psi_l(x)
\end{equation*}
where 
$\{{\xi_l}\}_{l=1}^\infty$ are independent and identically distributed mean zero, unit variance Gaussian random variables. 
This is simply the Karhunen-Lo\`eve representation of draws from
a mean zero Gaussian measure.
\end{remark}

The following result shows that 
the random variable $u$ is well-defined
and characterizes apsects of its regularity.

\begin{proposition}\cite{Las09} \label{p:LasFer}
Let $u$ be distributed according 
to a $(\kappa, X^{s,q})$ measure. 
The following are equivalent 
\begin{itemize}
	\item[i)] $\|u\|_{X^{t,q}}<\infty$ almost surely.
	\item[ii)] $\mathbb{E}(\exp(\alpha\|u\|^q_{X^{t,q}}))<\infty$ 
	               for any $\alpha\in (0,\kappa/2);$
	\item[iii)] $t<s-d/q$.
       
\end{itemize}
\end{proposition}

Part (ii) of the above proposition provides a Fernique-like result \cite{DaZa92} for Besov priors. Indeed
for $q=2$, this implies the Fernique theorem 
for all spaces $H^t$, $t<s-\frac{d}{q}$ on which the
measure is supported. 
The Fernique result for Gaussian measures, however, is much stronger: it shows that if $u\in X$ almost surely with respect to the Gaussian 
measure $\mu$ then $\bbE^{\mu} \exp(\epsilon\|u\|_X^2)<\infty$ for $\epsilon$ small enough.
Theorem \ref{t:Fer} below goes some way towards showing a 
similar result for Besov measures by extending (ii) to
$C^t$ spaces (see Remark \ref{r:Fer}).  Theorem  \ref{t:Fer} 
will also be useful in Section \ref{s:bBayes}
when we deal with inverse problems,
in the sense that it allows less restrictive conditions on the prior measure.

Before proving Theorem \ref{t:Fer} we make
a preliminary observation concerning the regularity of
$u$ given by \eqref{e:ranfun} in $C^t$ spaces.
With the same conditions on $s$ and $t$ as are assumed in Proposition \ref{p:LasFer}
one can show that $\bbE\|u\|_{C^t(\TT^d)}<\infty$.
Indeed, for any $\gamma\ge 1$, and $u$ given by \eqref{e:ranfun},
using the definition of the Besov norm we can write
\begin{align*}
\|u\|_{B^t_{\gamma q,\gamma q}}^{\gamma q}
&=(\tfrac{1}{\kappa})^\gamma
\sum_{l=1}^\infty l^{\frac{\gamma q t}{d}+\frac{\gamma q}{2}-1}
       l^{-\gamma q(\frac{s}{d}+ \frac{1}{2} -\frac{1}{q})}|\xi_l|^{\gamma q}.
\end{align*}
Noting that $\bbE |\xi_l|^{\gamma q}=C(\gamma)$ and 
the exponent of $l$ is smaller than $-1$ (since $t<s-d/p$), we have
\begin{align*}
\bbE\|u\|_{B^t_{\gamma q,\gamma q}}^{\gamma q}
=C(\gamma)(\tfrac{1}{\kappa})^\gamma
\sum_{l=1}^\infty l^{\frac{\gamma q}{d}(t-s)+\gamma-1}
\,\le\, C_1(\gamma).
\end{align*}
Now for a given $t<s-d/q$, choose $\gamma$ 
large enough so that $\frac{d}{\gamma q}<s-d/q-t$. 
Then the embedding $B^{t_1}_{\gamma q,\gamma q}\subset C^t$ 
for any $t_1$ satisfying $t+\frac{d}{\gamma q}<t_1<s-d/q$ \cite{Tri83} implies that
$\bbE\|u\|_{C^t(\TT^d)}<\infty$ and hence that $u \in C^t$
$\mu_0-$almost surely.
In the following theorem we show that, for
small enough $\alpha$, $\bbE \exp(\alpha\|u\|_{C^t(\TT^d)})<\infty$. For the proof we use the idea of the proof of a 
similar result for Radamacher series which appears
in Kahane \cite{Kah85}. It is also key in this proof that 
the wavelet basis is an unconditional 
basis for H\"older spaces \cite{Woj97}.

\begin{theorem}\label{t:Fer}
Let $u$ be a random function defined as in (\ref{e:ranfun})
with $q\ge 1$ and $s>d/q$. Then for any $t<s-d/q$
$$
\bbE(\exp(\alpha||u||_{C^{t}}))<\infty
$$
for all $\alpha\in (0,\kappa/(2\,r^*))$, 
with $r^*$ a constant depending on 
$q$, $d$, $s$ and $t$.
\end{theorem}

\begin{proof}
First, let $\kappa=1$. We have \cite{Tri83}
$$
\|u\|_{C^t}=\|u\|_{B^t_{\infty,\infty}}=\sup_{l\in \N} l^{(t-s)/d+1/q}|\xi_l|=\sup_{l\in \N} \lambda_l |\xi_l|
$$
with $\lambda_l =l^{(t-s)/d+1/q}$. Note that,
as shown above, $\|u\|_{C^t}<\infty$, $\mu_0$-almost surely.
Fix $r>0$ and let
\begin{subequations}
$$
A=\{\omega\in\Omega:\sup_{l\in \N} \lambda_l|\xi_l(\omega)|>r\},
$$
$$
B=\{\omega\in\Omega:\sup_{l\in \N} \lambda_l|\xi_l(\omega)|>2r\}.
$$
\end{subequations}
Consider the following disjoint partition of $A$
$$
A_m=\{\omega\in\Omega:\sup_{1\le l\le m-1} \lambda_l|\xi_l|\le r,\; \lambda_m|\xi_m|>r\},
$$
and define
$$
B_m=\{\omega\in\Omega:\sup_{l\ge m} \lambda_l|\xi_l|> 2r\}.
$$
We have
$$
\bbP(A_m\cap B)= \bbP(A_m\cap B_m).
$$
Noting that
\begin{align*}
B_m=C_m\cup B_{m+1}\quad\mbox{with}\quad C_m=\{\lambda_m|\xi_m|> 2r\},
\end{align*}
and
\begin{align*}
A_m=D_m\cap E_m,\quad\mbox{with}\quad &D_m=\{\omega\in\Omega:\sup_{1\le l\le m-1} \lambda_l|\xi_l|\le r\},\\
&E_m=\{\omega\in\Omega:\lambda_m|\xi_m|>r\},
\end{align*}
we can write
\begin{align}
\bbP(A_m\cap B)
&\le \bbP(A_m\cap C_m)+\bbP(A_m\cap B_{m+1})\nonumber\\
&= \bbP(D_m\cap C_m)+\bbP(A_m)\bbP(B_{m+1})\nonumber\\
&= \bbP(D_m)\bbP(C_m)+\bbP(A_m)\bbP(B_{m+1}).\label{e:parP}
\end{align}
We now show that $\bbP(C_m)\le (\bbP(E_m))^2$ for large enough $r$. First let $q>1$. We have, with $\hat r=r/\lambda_m$, $c_0=\int_{\R^d} \exp(-\frac{1}{2}|x|^q)\,\ud x$ and $c_d=2\pi^{d/2}/\Gamma(d/2)$,
\begin{align}
\bbP(C_m)
&=\frac{1}{c_0}\,\int_{\R^d} |x|\,\chi_{|x|>2\hat r}\,\e^{-\frac{1}{2}|x|^q}\ud x\nonumber\\
&=\frac{c_d}{c_0}\,\int_{2\hat r}^\infty \tilde\rho^d\,\e^{-\frac{1}{2}\tilde\rho^q}\ud \tilde\rho\nonumber\\
&=2^{d+1}\frac{c_d}{c_0}\,\int_{\hat r}^\infty \rho^d\,\e^{-\frac{2^q}{2}\rho^q}\ud \rho\nonumber\\
&=2^{d+1}\frac{c_d}{c_0}\,\int_{\hat r}^\infty \rho^d\,\e^{-\frac{1}{2}\rho^q}\,\e^{-\frac{2^q-1}{2}\rho^q}\ud \rho\nonumber\\
&\le 2^{d+1}\frac{c_d}{c_0}\,\e^{-\frac{2^q-1}{2}\hat r^q}\int_{\hat r}^\infty \rho^d\,\e^{-\frac{1}{2}\rho^q}\,\ud \rho\nonumber\\
&= 2^{d+1}\,\e^{-\frac{2^q-1}{2}\hat r^q}\,\bbP(E_m)\label{e:PCm0}.
\end{align}
Then one can show that $2^{d+1}\e^{-\frac{2^q-1}{2}\hat r^q}\le \bbP(E_m)$ for sufficiently large $\hat r$.
Indeed there exists $c_{q,d}$ depending only on $q$ and $d$ such that for $\rho\ge \hat r= c_{q,d}$ and $\nu\le 2^q-2$ we have 
\begin{align}\label{e:rbnd1}
\rho^d>2^{d+1}c_0\frac{\nu q+q}{2\,c_d}\,\rho^{q-1}\,\e^{-\nu\,\rho^q/2}.
\end{align}
Therefore we can write
\begin{align*}
\bbP(E_m)
&=\frac{c_d}{c_0}\, \int_{\hat r}^\infty \rho^d\,\e^{ -\frac{1}{2}\rho^q }\ud \rho\\
&> 2^{d+1}\int_{\hat r}^\infty \frac{\nu q+q}{2}\,\rho^{q-1}\,\e^{-(\nu+1)\,\rho^q/2}\ud \rho\\
&= 2^{d+1}\e^{-\frac{1+\nu}{2}\,\hat r^q} 
\,\ge\, 2^{d+1}\e^{-\frac{2^q-1}{2}\,\hat r^q} 
\end{align*}
Hence going back to (\ref{e:PCm0}) we have 
\begin{align}\label{e:PCm}
\bbP(C_m) \le (\bbP(E_m))^2.
\end{align}
For $q=1$ we can calculate $\bbP(E_m)$ and $\bbP(C_m)$ as follows 
$$
\bbP(E_m)=\frac{c_d}{c_0}\sum_{k=0}^d
       2^{k+1}\,{r}^{d-k}{d \choose k}\, \e^{-\frac{r}{2}},\quad \mbox{and}\quad 
  \bbP(C_m)=\frac{c_d}{c_0}\sum_{k=0}^d 2^{k+1}\,(2r)^{d-k}{d \choose k}\, \e^{-{r}}
$$
where $c_d/c_0=2^d$ when $q=1$. This readily shows (\ref{e:PCm}) for the case of $q=1$ as well.
Substituting (\ref{e:PCm}) in (\ref{e:parP}) we get
\begin{align*}
\bbP(A_m\cap B) 
&\le \bbP(D_m)(\bbP(E_m))^2+\bbP(A_m)\bbP(B_{m+1})\\
&= \bbP(D_m\cap E_m)\bbP(E_m)+\bbP(A_m)\bbP(B_{m+1})\\
&= \bbP(A_m)\bbP(E_m)+\bbP(A_m)\bbP(B_{m+1}).
\end{align*}
This, since  $E_m\subseteq A$ and $B_{m+1}\subseteq A$, implies that
$$
\bbP(A_m\cap B) \le 2 \bbP(A_m)\,\bbP(A).
$$
Writting the above inequality for $m=1,2,\dots$ and then adding them we obtain
\begin{align}\label{e:PBA}
\bbP(B)=\bbP(A\cap B)\le 2(\bbP(A))^2,
\end{align}
for $r>c_{q,d}$ (note that $\hat r=r/\lambda_m>c_q$ and $\lambda_m<1$).

Let $\mathscr{P}(\rho)=\bbP(\|u\|_{C^t}>\rho)$. We have
$$
\int_{C^t(\TT^d)}\e^{\epsilon\|u\|_{C^t}}\,\mu_0(\ud u)
=\int_{\Omega}\e^{\epsilon\|u(\omega)\|_{C^t}}\,\bbP(\ud\omega)
=-\int_0^\infty \e^{\epsilon\rho}\,\ud \mathscr{P}(\rho).
$$
Let $\mathscr{P}(r)=\bbP(A)=\beta$ and note that by (\ref{e:PBA}) we have
$$
\mathscr{P}(r)=\frac{1}{2}(2\beta),
\quad \mathscr{P}(2r)\le\frac{1}{2}(2\beta)^2,
\;\dots\,,\mathscr{P}(2^n r)\le\frac{1}{2}(2\beta)^{2^n},
$$
and therefore, choosing $r$ large enough so that $\beta<\frac{1}{2}$,
$$
-\int_{2^n r}^{2^{n+1}r}\e^{\epsilon\rho}\,\ud \mathscr{P}(\rho)
\le  \e^{2^{n+1}\epsilon r}\mathscr{P}(2^{n}r)
= \frac{1}{2}(2\beta)^{2^n}\,\e^{2^{n+1}\epsilon r}.
$$
Hence, letting $\beta=1/4$, we can write
\begin{align*}
-\int_0^\infty \e^{\epsilon\rho}\,\ud \mathscr{P}(\rho)
&\le-\int_0^r \e^{\epsilon\rho}\,\ud \mathscr{P}(\rho)+ \sum_{n=0}^\infty \frac{1}{2}(2)^{-2^n}\,\e^{2^{n+1}\epsilon r}\\
&=-\int_0^r \e^{\epsilon\rho}\,\ud \mathscr{P}(\rho)+ \sum_{n=0}^\infty\frac{1}{2} \e^{-2^{n}(\ln 2-2\epsilon r)}
\end{align*}
which is finite for $\epsilon<\frac{\ln 2}{2r}$. 
This proves the desired result for $\kappa=1$, once
we can identify the lower bound on $r$; a simple
rescaling gives the result for general $\kappa.$

We need to choose $r$ large enough so that $\bbP(A)=1/4$ and (\ref{e:rbnd1}) holds true as well. We have
\begin{align*}
\bbP(A)
&\le \sum_{j=1}^\infty \bbP\{\omega: \lambda_j|\xi_j(\omega)|> r\}
\le \frac{1}{c_{01}}\sum_{j=1}^\infty 
          \int_{\frac{r}{\lambda_j}}^\infty x^d\e^{-{x}/{2}}\,\ud x\\
&= \frac{1}{c_{01}}\sum_{j=1}^\infty \sum_{k=0}^d
       2^{k+1}\,\Big(\frac{r}{\lambda_j}\Big)^{d-k}{d \choose k}\, \e^{-\frac{r}{2\lambda_j}}\\
&\le \frac{1}{c_{01}} \sum_{k=0}^d 2^{k+1}{d \choose k}
\int_r^\infty x^{(d-k)(\frac{s-t}{d}-\frac{1}{q}) }
          \,\exp(-\frac{1}{2}x^{\frac{s-t}{d}-\frac{1}{q}})\,\ud x
\end{align*}
with $c_{01}=c_d^{-1}\int_{\R^d} \e^{-{|x|}/{2}}\,\ud x$ and 
noting that $\lambda_1=1$ and $\lambda_j=j^{{(t-s)}/{d}+{1}/{q}}$.
Now choose $r_1=r_1(s,t,d,q)$ such that
$$
\int_{r_1}^\infty x^{(d-k)(\frac{s-t}{d}-\frac{1}{q}) }
          \,\exp(-\frac{1}{2}x^{\frac{s-t}{d}-\frac{1}{q}})\,\ud x
< \frac{c_{01}}{4(d+1)}\frac{1}{2^{k+1}{d \choose k}},\quad\mbox{for}\quad k=0,\dots,d,
$$
and therefore $\bbP(A)<1/4$. 
To have (\ref{e:rbnd1}) true as well, we set
\begin{align*}
&\nu=0,\quad\mbox{and}\quad c_{q,d}=(2^dq\,c_0/c_d)^{1/(d-q+1)},
\quad\mbox{for}\quad1\le q<1+3/4\\
&\nu=1,\quad\mbox{and}\quad c_{q,d}=\max\{1,2^{d+2}q\, c_0/c_d\},
\quad\mbox{for}\quad q\ge 1+3/4.
\end{align*}
Since for $\kappa\neq 1$, $\epsilon=\alpha/\kappa$, this implies that $\alpha\le\kappa/(2\,r^*)$ with 
\begin{align}\label{e:rbnd}
r^*=(\ln2)\,\max\{r_1,c_{q,d}\}.
\end{align}
\end{proof}
 
 \begin{remark}\label{r:Fer}
Note that the bound on $\alpha$ in Theorem \ref{t:Fer} is
not sharp. Also, under the same 
condition as in Theorem \ref{t:Fer}, it is 
natural to expect a similar result to hold with power $q$ of 
 $\|u\|_{C^t}$ in the exponent and for the result to extend
to a norm in any space which has full measure; it would then be consistent
with the Gaussian Fernique theorem that arises when $q=2$
(see Theorem 2.6 in \cite{DaZa92}). 
However we have not been able to prove the result with this 
level of generality.
 \end{remark}
 \vskip.2cm

\section{Bayesian approach to inverse problems for functions}
\label{s:bBayes}

Recall the probabilistic inverse problem we introduced in Section \ref{s:intro}:
find the posterior distribution $\mu^y$ of $u\in X$, given a prior distribution
$\mu_0$ of $u\in X$, and $y\in Y$ given by \eqref{eq:one}
for a single realization of the random variable $\eta.$ 
We denote the distribution of $\eta$ on $Y$ by $\QQ_0(\ud y)$.
By (\ref{eq:one}) the distribution of $y$ given $u$ is known as well; we denote it by
$\Qu(\ud y)$ and, provided $\Qu$ is absolutely continuous
with respect to $\QQ_0$, we may 
define $\Phi:X\times Y\to \R$ so that 
\begin{align}\label{e:Phi}
\frac{\ud \Qu}{\ud \QQ_0}(y)=\exp \big(-\Phi(u;y) \big),
\end{align}
and
\begin{align}\label{e:asp0}
\int_Y \exp \big(-\Phi(u;y) \big) \,\QQ_0(\ud y)=1.
\end{align}
For instance if $\eta$ is a mean zero random Gaussian field on $Y$ 
with Cameron-Martin space $\big( E,\la\cdot,\cdot\ra_E,\|\cdot\|_{E}\big)$
then the Cameron-Martin formula (Proposition 2.24 
in \cite{DaZa92}) gives
\begin{align}
\Phi(u;y)
=\frac{1}{2}\|\cG(u)\|_E^2-\la y,\cG(u)\ra_E.
\end{align}
For finite-dimensional $Y=\bbR^K$, when $\eta$ has Lebesgue
density $\rho$, then we have
the identity $\exp(-\Phi(u;y))=\rho(y-\cG(u))/\rho(y).$

The previous section shows how to use wavelet or Fourier bases to
construct probability measures $\mu_0$
which are supported on a given Besov space $B^t_{qq}$ 
(and consequently on the H\"older space $C^t$).
Here we show how use of such priors $\mu_0$
may be combined with properties of
$\Phi$, defined above, to deduce the existence of
a well-posed Bayesian inverse problem. To this end
we assume the following conditions on $\Phi$:

\begin{asp} \label{asp1}
Let $X$ and $Y$ be Banach spaces.
The function $\Phi:X\times Y\to\R$ satisfies:
\begin{itemize}
\item[(i)] there is an $\alpha_1>0$ and for every $r>0$, an $M\in \R$, such that for all $u\in X$, and for all $y\in Y$ such that $||y||_Y<r,$
\begin{equation*}
\Phi(u,y) \geq M -\alpha_1||u||_{X};
\end{equation*}
\item[(ii)] for every $r>0$ there exists $K=K(r)>0$ such that for all $u\in X$, $y\in Y$ with $\max\{||u||_{X},||y||_Y\}<r$
\begin{equation*}
\Phi(u,y) \leq K;
\end{equation*}
\item[(iii)] for every $r>0$ there exists $L=L(r)>0$ such that for all $u_1,u_2\in X$ and $y\in Y$ with $\max\{||u_1||_{X},||u_2||_{X},||y||_Y\}<r$
\begin{equation*}
|\Phi(u_1,y)-\Phi(u_2,y)| \leq L||u_1-u_2||_{X};
\end{equation*}
\item[(iv)] there is an $\alpha_2>0$ and for every $r>0$ a $C\in \R$ such that for all $y_1,y_2\in Y$ with $\max\{||y_1||_Y,||y_2||_Y\}<r$ and for every $u\in X$
\begin{equation*}
|\Phi(u,y_1)-\Phi(u,y_2)| \leq \exp(\alpha_2||u||_{X}+C)||y_1-y_2||.
\end{equation*}
\end{itemize}
\end{asp}


\subsection{Well-defined and well-posed Bayesian inverse problems}

Recall the notation $\mu_0$ for the Besov prior measure defined
by\eqref{e:ranfun} and $\muy$ for the resulting posterior 
measure.  We now prove well-definedness 
and well-posedness of the 
posterior measure. The following theorems generalize 
the results of \cite{St10} from the case of Gaussian 
priors to Besov priors.

\begin{theorem}\label{t:welldb}
Let $\Phi$ satisfy (\ref{e:asp0}) and Assumption \ref{asp1}(i)--(iii).
Suppose that  for some $t<\infty$, $C^t$ is continuously embedded in $X$.
There exists $\kappa^*>0$ such that if
$\mu_0$ is a $(\kappa,X^{s,q})$ measure with
$s>t+\frac{d}{q}$ and $\kappa>\kappa^*$, 
then $\mu^y$ is absolutely continuous
with respect to $\mu_0$ and satisfies
\begin{equation}
\label{eq:radon}
\frac{d\mu^y}{d\mu_0}(u)= \frac{1}{Z(y)}\exp \bigl(-\Phi(u;y)\bigr),
\end{equation} 
with the normalizing factor
$Z(y)=\int_X \exp(-\Phi(u;y))\,\mu_0(\ud u)<\infty.$
The constant $\kappa^*=2\,c_e\,r^*\,\alpha_1$, where
$c_e$ is the embedding constant satisfying $\|u\|_X\le c_e\|u\|_{C^t}$, and $r^*$ is as in (\ref{e:rbnd}).
\end{theorem}

\begin{proof} 
Define $\pi_0(\ud u,\ud y)=\mu_0(\ud u)\otimes \QQ_0(\ud y)$
and $\pi(\ud u,\ud y)=\mu_0(\ud u)\Qu(\ud y)$.
Assumption \ref{asp1}(iii) gives continuity of $\Phi$ on $X$ and since $\mu_0(X)=1$ 
we have that $\Phi:X\to \R$ is $\mu_0$-measurable.  Therefore $\pi\ll\pi_0$ and $\pi$ has Radon-Nikodym derivative 
given by (\ref{e:Phi}) as (noting that by (\ref{e:asp0}) and since $\mu_0(X)=1$, we have 
$\int_{X\times Y}\exp\bigl(-\Phi(u;y)\bigr)\,\pi_0(\ud u,\ud y)=1$)
\begin{align*}
\frac{\ud \pi}{\ud\pi_0}(u;y)= \exp\bigl(-\Phi(u;y)\bigr).
\end{align*}
This then by Lemma 5.3 of \cite{HSV07}, implies that $\mu^y(du)=\pi(\ud u,\ud \xi |\xi=y)$ 
is absolutely continuous with respect to 
$\mu_0(y)=\pi_0(\ud u,\ud \xi |\xi=y)$, since 
$\pi_0$ is an independent product. This same
lemma also gives (\ref{eq:radon}) provided that the
normalization constant is positive, which we now establish. 
We note that all integrals over $X$ may be replaced
by integrals over $X^{t,q}$ for any $t<s-\frac{d}{q}$
since $\mu_0(X^{t,q})=1.$
First by Assumption \ref{asp1}(i) note that
there is $M=M(y)$ such that
\begin{align*}
Z(y)
&= \int_{X^{t,q}} \exp(-\Phi(u;y)) d\mu_0(u)\\
&\le \int_{X^{t,q}} \exp(\alpha_1\|u\|_{X}-M) d\mu_0(u)\\
&\le \int_{X^{t,q}} \exp(\alpha_1\,c_e\|u\|_{C^t}-M) d\mu_0(u)
\end{align*}
This upper bound is finite by Theorem \ref{t:Fer} since 
$\kappa>2\,c_e\,r^*\,\alpha_1$. 
We now prove that the normalisation 
constant does not vanish. 
Let $R=\bbE\|u\|_{X^{t,q}}$ noting that
$R \in (0,\infty)$ since $t<s-\frac{d}{q}.$ As
$\|u\|_{X^{t,q}}$ is a nonnegative
random variable we have that 
$\mu_0(\|u\|_{X^{t,q}}<R)>0$.  
Taking $r = \max\{\|y\|_Y,R\},$ 
Assumption \ref{asp1}(ii) gives
\begin{align*}
Z(y)&= \int_{X^{t,q}} \exp(-\Phi(u;y)) d\mu_0(u)\\
&\ge \int_{\|u\|_{X^{t,q}}<R} \exp(-K) d\mu_0(u) \\
&= \exp(-K) \mu_0(\|u\|_{X^{t,q}}<R) 
\end{align*}
which is positive. 
\end{proof}

We now show the well-posedness of the posterior measure $\mu^y$
with respect to the data $y$. Recall that 
the Hellinger metric  $\dhh$ is defined by 
$$\dhh(\mu,\mu')=\sqrt{\frac12 \int \left(
\sqrt\frac{\ud\mu}{\ud\nu}-\sqrt\frac{\ud\mu'}{\ud\nu}\right)^2
d\nu}.
$$
The Hellinger metric is independent of the choice of reference measure $\nu$, the measure with respect to which both $\mu$ 
and $\mu'$ are absolutely continuous. 
The posterior measure is Lipschitz with respect to data $y$,
in this metric.

\begin{theorem}\label{t:wellpb}
Let $\Phi$ satisfy (\ref{e:asp0}) and Assumption \ref{asp1}(i)--(iv).
Suppose that  for some $t<\infty$, $C^t$ is continuously embedded in $X$.
There exists $\kappa^*>0$ such that if
$\mu_0$ is a $(\kappa,X^{s,q})$ measure with
$s>t+\frac{d}{q}$ and $\kappa>\kappa^*$ then
$$
\dhh(\mu^y,\mu^{y'})\le C\,\|y-y'\|_Y
$$
where $C=C(r)$ with $\max\{\|y\|_Y,\|y'\|_Y\}\le r$.
The constant $\kappa^*=2\,c_e\,r^*(\alpha_1+2\alpha_2)$,  where
$c_e$ is the embedding constant satisfying $\|u\|_X\le c_e\|u\|_{C^t}$, and $r^*$ is as in (\ref{e:rbnd}).
\end{theorem}
\begin{proof}
As in Theorem \ref{t:welldb}, $Z(y),Z(y') \in (0,\infty)$. 
An application of the mean value theorem along with Assumption \ref{asp1}(i), and (iv) gives
\begin{align}
|Z(y)-Z(y')| 
&\le \int_{X^{t,q}} \left|\exp(-\Phi(u;y)) -\exp(-\Phi(u;y'))\right| d\mu_0(u)\nonumber\\
&\le \int_{X^{t,q}} \exp(\alpha_1\|u\|_{X}-M)|\Phi(u;y)-\Phi(u;y')|\ d\mu_0(u)\nonumber\\
&\le \int_{X^{t,q}} \exp\bigl((\alpha_1+
\alpha_2)\|u\|_{X}-M+C\bigl) \|y-y'\|_Y\,
d\mu_0(u)\nonumber\\
&\le C\|y-y'\|_Y,\label{normineq}
\end{align}
since $\|u\|_{X}\le c_e\|u\|_{C^t}$ and $c_e(\alpha_1+\alpha_2)<\kappa/(2r^*)$. Using the inequality $(a+b)^2\leq 2(a^2+b^2)$
\begin{align*}
2d_{\textrm{Hell}} &= \int_{X^{t,q}} \left(Z(y)^{-\frac{1}{2}}\exp(-\frac{1}{2}\Phi(u;y)-(Z(y')^{-\frac{1}{2}}\exp(-\frac{1}{2}\Phi(u;y')\right)^2 d\mu_0(u)\\
&\leq  I_1+ I_2
\end{align*}
where
\begin{align*}
I_1 &= \frac{2}{Z(y)} \int_{X^{t,q}} \left(\exp(-\frac{1}{2}\Phi(u;y))-\exp(-\frac{1}{2}\Phi(u;y')) \right)^2 d\mu_0(u)\\
I_2 &= 2|Z(y)^{-\frac{1}{2}}-Z(y')^{-\frac{1}{2}}|^2 
\int_{X^{t,q}} \exp(-\Phi(u;y'))d\mu_0(u)\\
&=2|Z(y)^{-\frac{1}{2}}-Z(y')^{-\frac{1}{2}}|^2 Z(y').
\end{align*}
Again, an application of the mean value theorem, and 
use of Assumptions \ref{asp1}(i) and (iv), gives
\begin{align*}
\frac{Z(y)}{2}I_1 
&\leq  \int_{X^{t,q}} \frac{1}{4}\exp(\alpha_1\|u\|_{X}-M)
              \exp(2\alpha_2\|u\|_{X}+2C)\|y-y'\|^2 \,\ud\mu_0(u)\\
&\leq  C\|y-y'\|^2_Y,
\end{align*}
since $c_e(\alpha_1+2\alpha_2)<\kappa/(2r^*)$. Recall that $Z(y)$ and $Z(y')$
are positive and bounded from above. Thus by the mean value theorem and (\ref{normineq})
\begin{eqnarray*}
I_2 =2Z(y')|Z(y)^{-\frac{1}{2}}-Z(y')^{-\frac{1}{2}}|^2  
\le C|Z(y)-Z(y')|^2 \leq C\|y-y'\|^2_Y.
\end{eqnarray*}
The result follows.
\end{proof}

\subsection{Approximation of the posterior}

Consider $\Phi^N$ to be an approximation of $\Phi$. 
Here we state a result which quantifies the effect of this approximation in the posterior measure in terms of
the aproximation error in $\Phi$. 

Define $\mu^{y,N}$ by
\begin{subequations}
\label{eq:muN}
\begin{equation}
\label{eq:muNa}
\frac{\ud\mu^{y,N}}{\ud\mu_0}(u)
=\frac{1}{Z^N(y)}\exp\bigl(-\Phi^N(u)\bigr),
\end{equation}
\begin{equation}
\label{eq:Anormz2}
Z^N(y)=\int_{X} \exp\bigl(-\Phi^N(u)\bigr) \ud\mu_0(u).
\end{equation}
\end{subequations}
We suppress the dependence of $\Phi$ and $\Phi^N$
on $y$ in this section as it is considered fixed.
%
%
\begin{theorem}
\label{t:wellp2b}
Assume that the measures $\mu$ and $\mu^N$ are both absolutely continuous with respect to $\mu_0$, and given by (\ref{eq:radon}) 
and (\ref{eq:muN}) respectively. 
Suppose that $\Phi$ and $\Phi^N$ satisfy
Assumption \ref{asp1}(i) and (ii), uniformly in $N$,
and that
there exist $\alpha_3\ge 0$ and $C\in\R$ such that
\begin{equation*}
|\Phi(u)-\Phi^N(u)| \leq \exp(\alpha_3\|u\|_{X}+C)\psi(N)
\end{equation*}
where $\psi(N)\rightarrow 0$ as $N\rightarrow \infty$.
Suppose that  for some $t<\infty$, $C^t$ is continuously embedded in $X$.
Let $\mu_0$ be a  $(\kappa,X^{s,q})$ measure with $s>t+\frac{d}{q}$ and
$\kappa>2\,c_e\,r^*(\alpha_1+2\alpha_3)$ where $r^*$ is as in (\ref{e:rbnd}) 
and $c_e$ is the embedding constant satisfying $\|u\|_X\le c_e\|u\|_{C^t}$.
Then there exists a constant independent of $N$ such that
\begin{equation*}
\dhh(\mu, \mu^N)\leq C\psi(N).
\end{equation*}
\end{theorem}

The proof is very similar to the proof 
of Theorem \ref{t:wellpb} and, in the Gaussian
case, is given in \cite{St10}; hence we omit it.

%
%


\section{Application to an elliptic inverse problem}\label{s:elliptic}

We consider the elliptic equation 
\begin{equation}\label{eq:epde}
\begin{array}{cl}
-\nabla\cdot \lt(\e^{u(x)}\nabla p(x)\rt)=f+\cdiv g,&x\in \TT^d,\\
\end{array}
\end{equation}
with periodic boundary conditions and with $\TT^d=(0,1]^d$, $d\le 3$, 
$p$, $u$ and $f$ scalar functions and $g$ a vector function on $\TT^d$. 
Given any $u\in L^\infty(\TT^d)$ we define $\lambda(u)$ and $\Lambda(u)$
by 
$$
\lambda(u)=\mathrm{ess}\inf_{x\in \TT^d}\e^{u(x)},\qquad
\Lambda(u)=\mathrm{ess}\sup_{x\in \TT^d}\e^{u(x)}.
$$
Where it causes no confusion we will simply write $\lambda$ or $\Lambda$.
Equation (\ref{eq:epde}) arises as a model for flow in a porous medium with $p$ the pressure
(or the head) and $\e^u$ the permeability (or the transmissivity); the velocity $v$
is given by the formula $v\propto -\e^u\nabla p$. 

Consider making noisy pointwise observations 
of the pressure field $p$.
We write the observations as
\begin{equation}
y_j=p(x_j)+\eta_j,\quad x_j\in\TT^d\quad j=1,\cdots,K.
\label{eq:obs}
\end{equation}
We assume, for simplicity,
that $\eta=\{\eta_j\}_{j=1}^K$ is a mean zero
Gaussian with covariance $\Gamma$. 
Our objective is to determine $u$ from $y=\{y_j\}_{j=1}^K \in \bbR^K$.
Concatenating the data, we have
$$y={\mathcal G}(u)+\eta,$$
with
\begin{equation}
\label{eq:bvpa}
\cG(u)=\bigl(p(x_1),\cdots,p(x_{\kk})\bigr)^T.
\end{equation}

In order to apply Theorem \ref{t:welldb}, \ref{t:wellpb} and \ref{t:wellp2b}
to the elliptic inverse problem
we need to prove certain properties of 
the forward operator $\cG$ given by (\ref{eq:bvpa}),
viewed as a mapping from a Banach space $X$ into $\R^m$.
The space $X$ must be chosen so that $C^t$ is continuously
embedded into $X$ and then the Besov prior $\mu_0$ chosen
with $s>t+\frac{d}{q}.$
In this section $|\cdot|$ stands for the Euclidean norm.
The following result is proved in \cite{DaSt11}.

\begin{proposition}\label{pr:G2}
Let $f\in L^r(\TT^d)$, $g\in L^{2r}(\TT^d)$.
Then for any $u\in L^\infty(\TT^d)$ there exists $C=C(K,d,r,\|f\|_{L^r},\|g\|_{L^{2r}})$ such that
$$
|\cG(u)|\le C\,\exp(\|u\|_{L^\infty(D)}).
$$
If $u_1,u_2\in C^t(D)$ for any $t>0$. Then,
for any $\epsilon>0$,
\begin{align*}
|\cG(u_1)-\cG(u_2)|
\le  C\,\exp\left(c
      \max\{\|u_1\|_{C^t(D)},\|u_2\|_{C^t(D)}\}
              \right)\|u_1-u_2\|_{L^\infty(D)}.
\end{align*}
with $C=C(K,d,t,\epsilon,\|f\|_{L^r},\|g\|_{L^{2r}})$
and  $c=4+(4+2d)/t+\epsilon.$ 
\end{proposition}
\vskip.2cm
Note that if instead of pointwise measurements of $p$, 
we consider the observations to be of the form $(l_1(p),\dots,l_{\kk}(p))$ 
where $l_j:H^1\to \R$, $j=1,\dots,\kk$, are bounded linear functionals,
then one can get similar boundedness and continuity properties of 
$\cG$ assuming $u$ to be only essentially bounded on $\TT^d$:
H\"older continuity is not needed.
However, since we construct the prior $\mu_0$ using a countable orthonormal
basis $\{\psi_l\}_{l\in\N}$, requiring the draws of $\mu_0$ to be bounded 
in $L^\infty(\TT^d)$ results, in any case,
in more regular draws which lie in a H\"older space. 
This is because $L^\infty(\TT^d)$ is not separable 
but any draw from 
$\mu_0$ can be expanded in $\{\psi_l\}_{l\in\N}$.
It is thus natural to consider the case of pointwise 
measurements since very similar arguments will also
deal with the case of measurements which are linear
functionals on $H^1$.



\subsection{Well-definedness and continuity of the posterior measure}

Now we can show the well-definedness of the posterior measure and its continuity with respect to the data for the elliptic problem. 
As we noted in Remark \ref{r:psi} by choosing $\{\psi_l\}_{l\in\N}$ of (\ref{e:ranfun})
as a wavelet or Fourier basis we can construct a Besov $(\kappa, B^s_{qq})$ or 
a Gaussian $(\kappa,H^s)$ prior measure ($B^s_{22}\equiv H^s$). We have the following theorem:

%
%

\begin{theorem}\label{thm:dwellposterior}
Consider the inverse problem for finding $u$ 
from noisy observations of $p$ in the form of (\ref{eq:obs}) 
and with $p$ solving (\ref{eq:epde}). Let $f\in L^r(\TT^d)$, $g\in L^{2r}(\TT^d)$ and
consider $\mu_0$ to be distributed as a Besov $(\kappa, B^s_{qq})$ prior 
with $1\le q<\infty$, $s>d/q$, $\kappa>0$ for $q=2$ and $\kappa>4r^*$ for $q\neq 2$ and $r^*$ as in (\ref{e:rbnd}).
Then the measure $\mu^y(\ud u)$ is absolutely continuous
with respect to $\mu_0$ with Radon-Nikodym derivative satisfying
$$
\frac{d\mu^y}{d\mu_0}(u) \propto \exp\Bigl(
-\frac{1}{2}\left|\Gamma^{-1/2}\bigl(y-{\cG}(u)\bigr)\right|^2+\frac{1}{2}|\Gamma^{-1/2}y|^2\Bigr).
$$
Furthermore, the posterior measure is continuous in the Hellinger metric with respect to the data
$$
\dhh(\mu^{y},\mu^{y'})\le C|y-y'|.
$$
\end{theorem}

\begin{proof}
Let $t<s-d/q$ and $X=C^t(\TT^d)$.
The function
$$
\Phi(u;y):=\frac{1}{2}\left|\Gamma^{-1/2}\bigl(y-{\cG}(u)\bigr)\right|^2-\frac{1}{2}|\Gamma^{-1/2}y|^2
$$ 
satisfies (\ref{e:asp0}) and Assumption \ref{asp1}(i) with $M=c\,r^2$, $c$ depending on $\Gamma$, and 
$\alpha_1=0$.
Using 
Proposition \ref{pr:G2}, Assumption \ref{asp1}(ii) and (iii) follow easily.
By Theorem \ref{t:Fer}, $\mu_0(C^t(D))=1$ for any $t$ such that $t<s-d/q$
and the absolute continuity of $\mu^y$ with respect to $\mu_0$ follows by Theorem \ref{t:welldb}.

We  note that
\begin{align*}
|\Phi(u;y_1)-\Phi(u;y_2)| &\le \frac12
\bigl|\Gamma^{-\frac12}\cG(u)\bigr|
\bigl|\Gamma^{-\frac12}\bigl(y_1-y_2\bigr)\bigr|\\
&\le c_1\exp\bigl(\|u\|_{X}\bigr)|y_1-y_2|
\end{align*}
Hence Assumption \ref{asp1} (iv) holds, and noting that $\alpha_2=1$, the continuity of $\mu^y$ with respect to the data follows
from Theorem \ref{t:wellpb}. %
\end{proof}

\subsection{Approximating the posterior measure}\label{sec:approx}

In this section, we consider the approximation of a
sufficiently regular $u$ in a 
finite-dimensional subspace of $L^2(\TT^d)$ and use the
Lipschitz continuity of $\cG$ 
together with Theorem \ref{t:wellp2b} to find an error estimate for
the corresponding approximate posterior measure.

Let $\{\psi_l\}_{l\in\N}$ be an orthonormal basis of $L^2(\TT^d)$ and
$W^N={\rm span}\{\psi_1,\dots,\psi_N\}$.
Denote the orthogonal projection of $L^2(\TT^d)$ onto $W^N$ by $P^N$ 
and let $\cG^N=\cG(P^N u)$. Define the approximated posterior measure $\mu^{y,N}$
by
\begin{align}\label{eq:muGN}
\frac{\ud\mu^{y,N}}{\ud\mu_0}(u)
=\frac{1}{Z^N(y)}\exp\bigl(-\frac{1}{2}|\Gamma^{-1/2}(y-\cG^N(u))|^2\bigr)
\end{align}
with $Z^N$ the normalizing factor.

Now we write $u\in L^2(\TT^d)$ in a wavelet basis: 
\begin{align}\label{e:uw}
u(x)
=u_1\phi(x)+\sum_{j=0}^\infty\sum_{(m,k)\in\Lambda_j}u_{m,k}\,\hat\psi_{m,k}(x).
\end{align}
In the above equation
$\phi$ is the scaling function for $L^2(\TT^d)$,
$k=(k_1,\dots,k_d)$,
$\Lambda_j=\{1,\dots,2^d-1\}\times\{0,\dots,2^j-1\}^d$,
and for each fixed $j$, 
$$\hat\psi_{m,k}(x)=2^{j/2}\sum_{n\in\Z^d}\bar\psi_m(2^j(x-\frac{k}{2^j}-n))$$
where $\bar\psi_m$ are the mother wavelet functions for $L^2(\R^d)$(see Chapter 3 of \cite{Mey92}). 
We also assume that the above wavelet basis is $r$-regular,
with $r$ sufficiently large (see Remark \ref{r:psi}).

We now impose one-dimensional indexing on the basis by setting $\psi_1=\phi$ 
and using the following numbering \cite{Las09,Mey92} for 
$\psi_l=\hat\psi_{m,k}$, $l>1$, 
\begin{align*}
\mbox{for }j=0&:\quad l=2,\dots,2^d,\\
\mbox{for }j=1&:\quad l=2^d+1,\dots,2^{2d},\\
&\;\vdots
\end{align*}
With this notation, the Karhunen-Lo\`eve expansion of a function $u$ drawn from
a  $(\kappa, B^s_{qq})$-Besov prior $\mu_0$ is the same as (\ref{e:ranfun}) and therefore 
the measure $\mu^{y,N}$ is an approximation to 
$\mu^y$ found by truncating the Karhunen-Lo\`eve expansion
of the prior measure to $N$ terms using the orthogonal projection $P^N$ defined above.

We have the following result on the convergence of 
$\mu^{y,N}$ to $\mu^y$ as $N\to\infty$:

\begin{theorem}\label{t:appw}
Consider the inverse problem of finding $u\in C^t(\TT^d)$,
with $t>0$,
from noisy observations of $p$ in the form of (\ref{eq:obs}) 
and with $p$ solving (\ref{eq:epde}) with periodic boundary conditions.
Assume that the prior $\mu_0$ is a $(\kappa, B^s_{qq})$ measure with $s>d/q+t$, $\kappa>0$ for $q=2$ and $\kappa>8r^*(2+(2+d)/t)$ otherwise.
Then
$$
\dhh(\mu^y,\mu^{y,N})\le C\,N^{-t/d}.
$$
\end{theorem}

We note that, although for a fixed $N$ the rate of convergence of
approximated posterior measure to $\mu^y$ in the wavelet case is smaller 
than that of the Fourier case, 
where $\dhh(\mu^y,\mu^{y,N})\le C\,N^{-t}$ (see \cite{DaSt11}),
one should take into account that we expect that the functions that solve the elliptic inverse problem of this section, have a more sparse expansion in a wavelet basis  compared to the Fourier basis
(see also section 9.4 of \cite{Dau92} or section 3.11 of \cite{Mey92}). \vskip.2cm

\begin{proof}[Proof of Theorem \ref{t:appw}]
Let $V_0$ and $W_j$ be the spaces spanned by 
$\{\phi\}$ and $\{\hat\psi_{m,k}\}_{(m,k)\in\Lambda_j}$  respectively.
Consider $Q_j$ to be the orthogonal projection in $L^2(\TT^d)$ onto $W_j$,
and $P_j$ the orthogonal projection of $L^2$ onto $\oplus_{k=1}^{j-1}W_k\oplus V_0$.
For any $f\in C^t(\TT^d)$ we can write \cite[Proposition 9.5 and 9.6]{Woj97}
$$
\|f-P_j f\|_{L^\infty}
\le C\,\sup_{0<|x-y|< 2^{-j}}\|f(x)-f(y)\|_{L^\infty}
\le C\,\,2^{-jt}\,\|f\|_{C^t}.
$$
Here and in the rest of this proof we represent any constant independent of $f$ and $j$
by $C$.
Using the above inequality, we have
\begin{align*}
\|Q_j f\|_{L^\infty}
&=\|P_{j+1}f-P_j f\|_{L^\infty}\\
&\le \|f-P_jf\|_{L^\infty}+\|f-P_{j+1}f\|_{L^\infty}
\,\le\, C\,2^{-jt}\|f\|_{C^t}.
\end{align*}
Hence
\begin{align*}
\|u-P^Nu\|_{L^\infty(D)}
&\le\sum_{j=J+1}^\infty\|Q_j u\|_{L^\infty}\\
&\le\, C\,\|u\|_{C^t}\sum_{j=J+1}^\infty 2^{-jt}
=C\,\|u\|_{C^t}\,2^{-(J+1)t}\sum_{j=0}^\infty 2^{-jt} \\
&\,\le\, C\,\|u\|_{C^t}\, 2^{-(J+1)t} 
\,\le\, C\,\|u\|_{C^t}\, N^{-t/d}.
\end{align*}
By Proposition \ref{pr:G2} we have 
$$
|\Phi(u)-\Phi(P^N u)|\le C\exp\left( c_1\,\|u\|_{C^t(D)}\right)N^{-t/d},
$$
with  $c_1>4+(4+2d)/t$. The result therefore follows by Theorem \ref{t:wellp2b}.
\end{proof}

\begin{remark}
Let $W^\perp$ be the orthogonal complement of $W^N$ in $L^2(\TT^d)$.
Since $\mu_0$ is defined by 
the Karuhnen-Lo\`eve expansion of its draws
as in (\ref{e:ranfun}) using $\{\psi_l\}_{l\in\N}$, 
it factors as the product of two measures
$\mu_0^N\otimes\mu_0^\bot$ on $W^N\oplus W^\bot$.
Since $\cG^N(u)=\cG(P^N u)$ depends only on $P^N u$,
we may factor $\mu^{y,N}$ as $\mu^{y,N}=\mu^N\otimes\mu^\bot$
where $\mu^N$ satisfies
\begin{align}\label{eq:muGN1}
\frac{\ud\mu^{N}}{\ud\mu_0^N}(u)
=\frac{1}{Z^N}\exp\bigl(-\frac{1}{2}|\Gamma^{-1/2}(y-\cG^N(u))|^2\bigr)
\end{align}
and $\mu^\perp=\mu_0^\perp$.
With this definition of $\mu^N$ as a measure on the finite dimensional space 
$W^N$ and having the result of Theorem \ref{t:appw} one can estimate the following weak errors (see Theorem 2.6 of \cite{DaSt11}):
$$
\|\bbE^{\mu^y}p-\bbE^{\mu^{N}}p^N\|_{L^\infty(\TT^d)}\le C\,N^{-t/d},
$$
$$
\|\bbE^{\mu^y}(p-\bar{p})\otimes (p-\bar{p})
-\bbE^{\mu^{N}}(p^N-\bar{p}^N)\otimes (p^N-\bar{p}^N)\|_S
\le C\,N^{-t/d},
$$
with $p^N$ the solution of (\ref{eq:epde}) for $u=P^N u$, $\bar{p}=\bbE^{\mu^y} p$, $\bar{p}^N=\bbE^{\mu^{N}} p^N$
and $S=\mathcal{L}(H^1(\TT^d),H^1(\TT^d))$.

\end{remark}

\section{Conclusion}\label{sec:conclu}

We used a Bayesian approach \cite{ks04}
to find a well-posed probabilistic formulation 
of the solution to the
inverse problem of finding a function $u$ from noisy measurements of a known function
$\cG$ of $u$. The philosophy underlying this approach is
that formulation of the problem on function space leads
to greater insight concerning both the structure of
the problem, and the development of effective algorithms
to probe it. In particular it leads to the formulation of
problems and algorithms which are robust under mesh-refinement
\cite{St10}.
Motivated by the sparsity promoting features of the wavelet bases for many classes of functions appearing in 
applications, we studied the use of the Besov priors 
introduced in \cite{Las09} within the Bayesian formalism. 

Our main goal has been to generalize the results of \cite{St10}
on well-definedness and well-posedness of the posterior measure 
for the Gaussian priors, to the case of Besov priors. 
We showed that if the operator $\cG$ satisfies certain regularity conditions on the Banach space $X$, then provided that the Besov prior is chosen appropriately, 
the posterior measure over $X$ is well-defined and well-posed (Theorems \ref{t:welldb} and \ref{t:wellpb}). Using
the well-posedness of the posterior on the infinite-dimensional space $X$, we then studied the convergence of the appropriate finite-dimensional 
approximations of the posterior.
In finding the required conditions on $\cG$,
it is essential to know which functions of $u$ have finite integral with respect to the Besov prior $\mu_0$. In other words we need a result similar to the Fernique theorem 
from Gaussian measures, for the Besov case. 
A Fernique-like result for H\"older norms
is proved in Theorem \ref{t:Fer},
and may be of independent interest. 

As an application of these results, we have considered the problem of finding the diffusion coefficient of an elliptic partial differential equation from noisy measurements of its solution. We have found the conditions on the Besov prior which make the
Bayesian formalism well-posed for this problem. 
We have also considered the approximation of the posterior measure on 
a finite-dimensional space spanned by finite number of elements of the same wavelet basis used in constructing the prior measure, and quantified the error incurred by such an approximation.

A question left open by the analysis in this paper is how to
extract information from the posterior measure. Typical information
desired in applications involves the computation of expectations
with respect to the posterior. A natural approach to this is
through the use of Markov chain-Monte Carlo (MCMC). For Gaussian priors
there has been considerable recent effort to develop new MCMC methods
which are discretization invariant \cite{BRSV08,CDS11} in that
they are well-defined in the
infinite dimensional limit; it would be interesting to extend
this methodology to Besov priors. In the meantime the analysis
of standard Random Walk and Langevin algorithms in \cite{BRS}
applies to the posterior meausures constructed in this paper and
quantifies the increase in computational cost incurred as
dimension increases, resulting from the fact that the
infinite dimensional limit is not defined for these standard algorithms.
A second approach to integration in high dimensions is via polynomial
chaos approximation \cite{SpG89} and a recent application of this
approach to an inverse problem may be found in \cite{ScSt11}. 
A third approach is the use of quasi-Monte Carlo methods; see \cite{caf98}.
It would be of interest to study the application of all of these
methodologies to instances of the Besov-prior inverse problems constructed
in this paper.

\section*{Acknowledgements}
MD and AMS are grateful to the EPSRC (UK) and ERC for financial support.

\bibliographystyle{99}

\end{document}